\documentclass{amsart}
\usepackage{mathptmx}
\usepackage{amscd,amssymb,enumerate,bm}

\newtheorem{theorem}{Theorem}[section]
\newtheorem{lemma}[theorem]{Lemma}
\newtheorem{corollary}[theorem]{Corollary}

\theoremstyle{definition}

\newtheorem{example}[theorem]{Example}

\newtheorem{conjecture}[theorem]{Conjecture}
\numberwithin{equation}{theorem}

\def\image{\operatorname{image}}
\def\Ass{\operatorname{Ass}}
\def\End{\operatorname{End}}
\def\Hom{\operatorname{Hom}}

\def\Proj{\operatorname{Proj}}

\def\bsf{{\boldsymbol{f}}}

\def\fraka{\mathfrak{a}}

\def\frakp{\mathfrak{p}}

\def\FF{\mathbb{F}}
\def\NN{\mathbb{N}}
\def\PP{\mathbb{P}}
\def\QQ{\mathbb{Q}}
\def\ZZ{\mathbb{Z}}

\def\calD{\mathcal{D}}
\def\calF{\mathcal{F}}
\def\calM{\mathcal{M}}

\def\phi{\varphi}
\def\del{\partial}
\def\tilde{\widetilde}

\def\to{\longrightarrow}
\def\mapsto{\longmapsto}
\renewcommand{\mod}{\,\operatorname{mod}\,}

\begin{document}
\title[Local cohomology of a smooth $\mathbb{Z}$-algebra]{Local cohomology modules of a smooth $\mathbb{Z}$-algebra have finitely many associated primes}

\author[Bhatt]{Bhargav Bhatt}
\address{School of Mathematics, Institute for Advanced Study, Einstein Drive, Princeton,
\newline NJ~08540, USA}
\email{bhargav.bhatt@gmail.com}

\author[Blickle]{Manuel Blickle}
\address{Institut f\"ur Mathematik, Fachbereich 08, Johannes Gutenberg-Universit\"at Mainz,
\newline 55099~Mainz, Germany}
\email{blicklem@uni-mainz.de}

\author[Lyubeznik]{Gennady Lyubeznik}
\address{Department of Mathematics, University of Minnesota, 127 Vincent Hall, 206 Church~St.,
\newline Minneapolis, MN~55455, USA}
\email{gennady@math.umn.edu}

\author[Singh]{Anurag K. Singh}
\address{Department of Mathematics, University of Utah, 155 South 1400 East, Salt Lake City,
\newline UT~84112, USA}
\email{singh@math.utah.edu}

\author[Zhang]{Wenliang Zhang}
\address{Department of Mathematics, University of Nebraska, 203 Avery Hall, Lincoln,
\newline NE~68588, USA}
\email{wzhang15@unl.edu}

\thanks{B.B.~was supported by NSF grants DMS~1160914 and DMS 1128155, M.B.~by DFG grants SFB/TRR45 and his Heisenberg Professorship, G.L.~by NSF grant DMS~1161783, A.K.S.~by NSF grant DMS~1162585, and W.Z.~by NSF grant DMS~1068946. A.K.S. thanks Uli Walther for several valuable discussions. The authors are grateful to the American Institute of Mathematics (AIM) for supporting their collaboration. All authors were also supported by NSF grant~0932078000 while in residence at MSRI}


\begin{abstract}
Let $R$ be a commutative Noetherian ring that is a smooth $\ZZ$-algebra. For each ideal $\fraka$ of $R$ and integer $k$, we prove that the local cohomology module $H^k_\fraka(R)$ has finitely many associated prime ideals. This settles a crucial outstanding case of a conjecture of Lyubeznik asserting this finiteness for local cohomology modules of all regular rings.
\end{abstract}
\maketitle

\section{Introduction}

A question of Huneke~\cite[Problem~4]{Huneke:Sundance} asks whether local cohomology modules of Noetherian rings have finitely many associated prime ideals. The answer is negative in general: the first counterexample was given by Singh~\cite[\S\!~4]{Singh:MRL}, and further counterexamples were obtained by Katzman~\cite{Katzman} and Singh and Swanson~\cite{SS:IMRN}. 

However, there are several affirmative answers: by work of Huneke and Sharp~\cite{HS:TAMS}, for regular rings~$R$ of prime characteristic; by work of Lyubeznik, for regular local and affine rings of characteristic zero~\cite{Lyubeznik:Invent}, and for unramified regular local rings of mixed characteristic~\cite{Lyubeznik:Comm}; for a partial result in the case of  ramified regular local rings, see N\'u\~nez-Betancourt~\cite{Nunez}. These results support Lyubeznik's conjecture,~\cite[Remark~3.7]{Lyubeznik:Invent}:

\begin{conjecture}
If $R$ is a regular ring, then each local cohomology module $H_\fraka^k(R)$ has finitely many associated prime ideals.
\end{conjecture}

While the counterexamples from~\cite{Katzman} and~\cite{SS:IMRN} are for rings containing a field, the local cohomology module with infinitely many associated primes from~\cite{Singh:MRL} has the form~$H^k_\fraka(R)$ where $R$ is a hypersurface over the integers; in this example, $H^k_I(R)$ has nonzero $p$-torsion for each prime integer $p$. A major stumbling block in making progress with Lyubeznik's conjecture for rings not containing a field was the possibility of $p$-torsion for infinitely many prime integers $p$. The key point in this paper is to show that for a smooth $\ZZ$-algebra~$R$, the $p$-torsion of each local cohomology module $H^k_\fraka(R)$ can be controlled; this allows us to settle an important case of Lyubeznik's conjecture: 

\begin{theorem}
\label{theorem:main}
Let $R$ be a smooth $\ZZ$-algebra, $\fraka$ an ideal of $R$, and $k$ a nonnegative integer. Then the set of associated primes of the local cohomology module $H^k_\fraka(R)$ is finite.
\end{theorem}

Our proof uses $\calD$-modules over $\ZZ$, $\FF_p$, and $\QQ$, along with the theory of~$\calF$-modules developed in \cite{Lyubeznik:Crelle}. The relevant results are reviewed in~\S\!~\ref{section:DF}. A crucial step in the proof is to relate the integer torsion in a local cohomology module to the integer torsion in a Koszul cohomology module; since the latter is finitely generated, it has $p$-torsion for at most finitely many $p$. The proof of the main theorem occupies \S\!~\ref{section:proof}.

Our techniques work somewhat more generally: in \S\!~\ref{section:dedekind} we indicate the changes that need to be made to tackle the case where $R$ is a smooth algebra over a Dedekind domain, all of whose residue fields at nonzero prime ideals are of characteristic $p$. Our techniques are also sufficient to give a new and much simpler proof of the case of an unramified regular local ring of mixed characteristic, originally obtained by Lyubeznik in \cite{Lyubeznik:Comm}.

\section{$\calD$-modules and $\calF$-modules}
\label{section:DF}

\subsection{$\calD$-modules}
\label{subsection:D}

Let $R$ be a commutative ring. \emph{Differential operators} on $R$ are defined inductively as follows: for each $r\in R$, the multiplication by $r$ map $\tilde{r}\colon R\to R$ is a differential operator of order $0$; for each positive integer $n$, the differential operators of order less than or equal to $n$ are those additive maps $\delta\colon R\to R$ for which the commutator
\[
[\tilde{r},\delta]\ =\ \tilde{r}\circ\delta-\delta\circ\tilde{r}
\]
is a differential operator of order less than or equal to $n-1$. If $\delta$ and $\delta'$ are differential operators of order at most $m$ and $n$ respectively, then $\delta\circ\delta'$ is a differential operator of order at most $m+n$. Thus, the differential operators on~$R$ form a subring $\calD(R)$ of $\End_\ZZ(R)$.

When $R$ is an algebra over a commutative ring $A$, we define $\calD(R,A)$ to be the subring of~$\calD(R)$ consisting of differential operators that are $A$-linear. Note that $\calD(R,\ZZ)=\calD(R)$; if $R$ is an algebra over a perfect field $\FF$ of prime characteristic, then $\calD(R,\FF)=\calD(R)$, see, for example, \cite[Example~5.1~(c)]{Lyubeznik:Crelle}.

By a $\calD(R,A)$-module, we mean a \emph{left} $\calD(R,A)$-module. Since $\calD(R,A)\subseteq\End_A(R)$, the ring $R$ has a natural $\calD(R,A)$-module structure. Using the quotient rule, localizations of~$R$ also carry a natural $\calD(R,A)$-structure. Let $\fraka$ be an ideal of $R$. The \v Cech complex on a generating set for $\fraka$ is a complex of $\calD(R,A)$-modules; it then follows that each local cohomology module $H^k_{\fraka}(R)$ is a $\calD(R,A)$-module.

More generally, if $M$ is a $\calD(R,A)$-module, then each local cohomology module $H^k_{\fraka}(M)$ is also a $\calD(R,A)$-module, see~\cite[Examples 2.1~(iv)]{Lyubeznik:Invent} or \cite[Example~5.1~(b)]{Lyubeznik:Crelle}.

If $R$ is a polynomial or formal power series ring in variables $x_1,\dots,x_d$ over a commutative ring $A$, then $\frac{1}{{t_i}!}\frac{\del^{t_i}}{\del x_i^{t_i}}$ can be viewed as a differential operator on $R$ even if the integer $t_i!$ is not invertible. In each of these cases, $\calD(R,A)$ is the free~$R$-module with basis
\[
\frac{1}{{t_1}!}\frac{\del^{t_1}}{\del x_1^{t_1}}\ \cdots\ \frac{1}{{t_d}!}\frac{\del^{t_d}}{\del x_d^{t_d}}
\qquad\text{ for } \ (t_1,\dots,t_d)\in\NN^d\,,
\]
see~\cite[Th\'eor\`eme~16.11.2]{EGA4}. If $B$ is an $A$-algebra, it follows that
\[
\calD(R,A)\otimes_AB\ \cong\ \calD(R\otimes_AB\,,\,B)\,.
\]
Specifically, for each element $a\in A$, one has
\begin{equation}
\label{equation:dmodp}
\calD(R,A)/a\calD(R,A)\ \cong\ \calD(R/aR\,,\,A/aA)\,.
\end{equation}

To obtain analogous results for any smooth $A$-algebra, we use an alternative description of $\calD(R,A)$ from~\cite[16.8]{EGA4}: consider the left $R\otimes_AR$-module structure on $\End_A(R)$ under which $r\otimes s$ acts on $\delta$ to give the endomorphism $\tilde{r}\circ\delta\circ\tilde{s}$ where, as before, $\tilde{r}$ denotes the multiplication by $r$ map. Set $\Delta_{R/A}$ to be the kernel of the ring homomorphism $R\otimes_AR\to R$ with $r\otimes s\mapsto rs$. The ideal $\Delta_{R/A}$ is generated by elements of the form $r\otimes\!1-1\otimes r$. Since
\[
(r\otimes\!1-1\otimes r)(\delta)\ =\ [\tilde{r},\delta]\,,
\]
it follows that an element $\delta$ of $\End_A(R)$ is a differential operator of order at most $n$ precisely if it is annihilated by $\Delta^{n+1}_{R/A}$. By~\cite[Proposition~16.8]{EGA4}, the $A$-linear differential operators on $R$ of order at most~$n$ correspond to
\[
\Hom_{R\otimes_AR}\big((R\otimes_AR)/\Delta^{n+1}_R,\,\End_A(R)\big)\ \cong\ \Hom_R(P^n_{R/A},\,R)\,,
\]
where
\[
P^n_{R/A}\ =\ (R\otimes_AR)/\Delta^{n+1}_R\,,
\]
viewed as a left $R$-module via $r\mapsto r\otimes\!1$.

A ring $R$ is said to be \emph{smooth} over $A$ if $R$ is a finitely presented and flat $A$-algebra, such that for each prime ideal $\frakp$ of $A$, the fiber $R_\frakp/\frakp R_\frakp$ is geometrically regular over $A_\frakp/\frakp A_\frakp$. In this situation, we have:

\begin{lemma}
\label{lemma:dmodp}
If $R$ is a smooth $A$-algebra, then for each $A$-algebra $B$ one has
\[
\calD(R,A)\otimes_AB\ \cong\ \calD(R\otimes_AB\,,\,B)\,.
\]
\end{lemma}

\begin{proof}
Since $R$ is $A$-smooth, the $R$-module $P^n_{R/A}$ is locally free of finite rank by~\cite[Proposition~16.10.2]{EGA4}. It follows that 
\begin{equation}
\label{equation:lemma}
\Hom_R(P^n_{R/A}\,,\,R)\otimes_AB\ \cong\ \Hom_{R_B}(P^n_{R/A}\otimes_AB\,,\,R_B)\,,
\end{equation}
where $R_B=R\otimes_AB$. Since $R$ is flat over $A$, one also has $\Delta_{R/A}\otimes_AB\cong\Delta_{R_B/B}$. Tensoring the exact sequence
\[
\CD
0@>>>\Delta^{n+1}_{R/A}@>>>R\otimes_AR@>>>P^n_{R/A}@>>>0
\endCD
\]
with $B$, one obtains the first row of the commutative diagram
\[
\CD
@.\Delta^{n+1}_{R/A}\otimes_AB@>>>R_B\otimes_BR_B@>>>P^n_{R/A}\otimes_AB@>>>0\phantom{\,.}\\
@.@VVV@|@VVV\\
0@>>>\Delta^{n+1}_{R_B/B}@>>>R_B\otimes_BR_B@>>>P^n_{R_B/B}@>>>0\,.
\endCD
\]
The vertical map on the left is surjective, which gives $P^n_{R/A}\otimes_AB\cong P^n_{R_B/B}$. Combining this with~\eqref{equation:lemma}, we get the desired isomorphism
\[
\Hom_R(P^n_{R/A}\,,\,R)\otimes_AB\ \cong\ \Hom_{R_B}( P^n_{R_B/B}\,,\,R_B)\,.\qedhere
\]
\end{proof}

\subsection{$\calF$-modules}
\label{subsection:F}

We next review some aspects of the theory of $\calF$-modules, developed by Lyubeznik in~\cite{Lyubeznik:Crelle}. Let $R$ be an $F$-finite regular ring of prime characteristic $p$. For each positive integer~$e$, define $R^{(e)}$ to be the $R$-bimodule that agrees with $R$ as a left $R$-module, and that has the right $R$-action
\[
r'r=r^{p^e}r'\qquad\text{ for }r\in R\text{ and }r'\in R^{(e)}\,.
\]
For an $R$-module $M$, define $F(M)=R^{(1)}\otimes_R M$; we view this as an $R$-module via the left $R$-module structure on $R^{(1)}$.

An \emph{$\calF$-module} is an $R$-module $\calM$ with an $R$-module isomorphism $\theta\colon\calM\to F(\calM)$. The ring $R$ has a natural $\calF$-module structure, and so does each local cohomology module~$H^k_\fraka(R)$, see~\cite[Example~1.2]{Lyubeznik:Crelle}. An $\calF$-module carries a natural $\calD(R)$-module structure by~\cite[pages~115--116]{Lyubeznik:Crelle}. When the $\calF$-module $\calM$ is the ring $R$, a localization of $R$, or a local cohomology module $H^k_\fraka(R)$, the usual $\calD(R)$-module structure on~$\calM$ agrees with the one induced via the $\calF$-module structure; see~\cite[Example~5.2~(c)]{Lyubeznik:Crelle}.

A \emph{generating morphism} for an $\calF$-module $\calM$ is an $R$-module map $\beta\colon M\to F(M)$ such that $\calM$ is the direct limit of the top row of the commutative diagram
\[
\CD
M@>\beta>>F(M)@>F(\beta)>>F^2(M)@>F^2(\beta)>>\cdots\\
@V\beta VV @VF(\beta)VV @VF^2(\beta)VV\\
F(M)@>F(\beta)>>F^2(M)@>F^2(\beta)>>F^3(M)@>F^3(\beta)>>\cdots\,.
\endCD
\]
Note that the direct limit of the bottom row is $F(\calM)$, and that the vertical maps induce the isomorphism $\theta\colon\calM\to F(\calM)$. If $\beta\colon M\to F(M)$ is a generating morphism for $\calM$, then the image of $M$ in $\calM$ generates $\calM$ as a $\calD(R)$-module by~\cite[Corollary~4.4]{ABL}; this is a key ingredient in the proof of our main result.

\subsection{Koszul and local cohomology}
\label{subsection:koszul}

Given $f\in R$, there is a map of complexes
\[
\CD
K^\bullet(f;R) @.\qquad=\qquad @. 0@>>>R@>f>>R@>>>0\phantom{\,,}\\
@VVV @. @. @| @VVf^{p-1}V\\ 
K^\bullet(f^p;R) @.=@. 0@>>>R@>f^p>>R@>>>0\phantom{\,,}\\
@VVV @. @. @| @VV\frac{1}{f^p}V\\ 
C^\bullet(f;R) @.=@. 0@>>>R@>>>R_f@>>>0\,,
\endCD
\]
where $K^\bullet$ denotes the Koszul complex, and $C^\bullet$ the \v Cech complex. Let $\bsf=f_1,\dots,f_t$ be a sequence of elements of $R$. Regarding $K^\bullet(\bsf;R)$ and $C^\bullet(\bsf;R)$ as the tensor products
\[
K^\bullet(f_1;R)\otimes\cdots\otimes K^\bullet(f_t;R)
\quad\text{ and }\quad
C^\bullet(f_1;R)\otimes\cdots\otimes C^\bullet(f_t;R)
\]
respectively, one obtains a map of complexes
\[
\CD
K^\bullet(\bsf;R)@>>>K^\bullet(\bsf^p;R)@>>>C^\bullet(\bsf;R)\,,
\endCD
\]
and induced maps on cohomology modules
\[
\CD
H^k(\bsf;R)@>\beta>>H^k(\bsf^p;R)@>>>H^k_\fraka(R)\,,
\endCD
\]
where $\fraka$ is the ideal generated by $\bsf$. By~\cite[Proposition~1.11~(b)]{Lyubeznik:Crelle}, the map~$\beta$ is a generating homomorphism for the local cohomology module $H^k_\fraka(R)$; hence the image of~$H^k(\bsf;R)$ in $H^k_\fraka(R)$ generates $H^k_\fraka(R)$ as a $\calD(R)$-module, as mentioned at the end of~\S\!~\ref{subsection:F}.

\section{The main theorem}
\label{section:proof}

We prove the following result that subsumes~Theorem~\ref{theorem:main}.

\begin{theorem}
\label{theorem:integer}
Let $R$ be a smooth $\ZZ$-algebra, and $\fraka$ an ideal of $R$ generated by elements $\bsf=f_1,\dots,f_t$. Let $k$ be a nonnegative integer.
\begin{enumerate}[\quad\rm(1)]
\item If a prime integer is a nonzerodivisor on the Koszul cohomology module $H^k(\bsf;R)$, then it is a nonzerodivisor on the local cohomology module $H^k_\fraka(R)$.

\item All but finitely many prime integers are nonzerodivisors on $H^k_\fraka(R)$.

\item The set of associated primes of the $R$-module $H_\fraka^k(R)$ is finite.
\end{enumerate}
\end{theorem}

\begin{proof}
Let $p$ be a prime integer. The exact sequence
\[
\CD
0@>>>R@>p>>R@>>>R/pR@>>>0
\endCD
\]
induces an exact sequence of Koszul cohomology modules and an exact sequence of local cohomology modules; these fit into a commutative diagram:
\[
\CD
@>>>H^{k-1}(\bsf;R)@>\pi>>H^{k-1}(\bsf;R/pR)@>>>H^k(\bsf;R)@>p>>H^k(\bsf;R)@>>>\\
@. @V\alpha' VV @VV\alpha V @VVV @VVV \\
@>>>H^{k-1}_\fraka(R)@>\phi>>H^{k-1}_\fraka(R/pR)@>d>>H^k_\fraka(R)@>p>>H^k_\fraka(R)@>>>
\endCD
\]
The bottom row is a complex of $\calD(R)$-modules; in particular, $\phi\big(H^{k-1}_\fraka(R)\big)$ is a $\calD(R)$-submodule of $H^{k-1}_\fraka(R/pR)$. As $\phi\big(H^{k-1}_\fraka(R)\big)$ is annihilated by $p$, it has a natural structure as a module over the ring $\calD(R)/p\calD(R)$, which equals~$\calD(R/pR)$ by Lemma~\ref{lemma:dmodp}. Similarly,
\begin{equation}
\label{eqn:image:d}
\CD
H^{k-1}_\fraka(R/pR)@>d>>\image(d)
\endCD
\end{equation}
is a map of $\calD(R/pR)$-modules.

(1) Suppose $p$ is a nonzerodivisor on $H^k(\bsf;R)$. Then the map $\pi$ is surjective; we need to prove that $p$ is a nonzerodivisor on $H^k_\fraka(R)$, equivalently, that $\phi$ is surjective.

By~\S\!~\ref{subsection:koszul}, the image $M$ of $\alpha$ generates $H^{k-1}_\fraka(R/pR)$ as a $\calD(R/pR)$-module. As~$\pi$ is surjective, $M$ is also the image of $\alpha\circ\pi=\phi\circ\alpha'$. It follows that
\[
M\ \subseteq\ \phi\big(H^{k-1}_\fraka(R)\big)\,.
\]
But $\phi\big(H^{k-1}_\fraka(R)\big)$ is a $\calD(R/pR)$-submodule of $H^{k-1}_\fraka(R/pR)$ that contains $M$. Hence
\[
\phi\big(H^{k-1}_\fraka(R)\big)\ =\ H^{k-1}_\fraka(R/pR)\,,
\]
i.e., $\phi$ is surjective, as desired.

(2) Since $H^k(\bsf;R)$ is a finitely generated $R$-module, it has finitely many associated prime ideals. These finitely many prime ideals contain at most finitely many prime integers; all other prime integers are nonzerodivisors on $H^k(\bsf;R)$, and hence on $H^k_\fraka(R)$ by (1).

(3) We have proved that the set $\Ass_\ZZ H^k_\fraka(R)$ is finite; let $\frakp$ be an element of this set. It suffices to show that there are at most finitely many elements of~$\Ass_R H^k_\fraka(R)$ that lie over~$\frakp$.

If $\frakp$ is the zero ideal, then each associated prime of $H^k_\fraka(R)$ lying over $\frakp$ is the contraction of an associated prime of
\[
H^k_\fraka(R)\otimes_{\ZZ}\QQ\ =\ H^k_\fraka\big(R\otimes_\ZZ\QQ\big)
\]
as an $R\otimes_\ZZ\QQ$-module. Since $R\otimes_\ZZ\QQ$ is a regular finitely generated $\QQ$-algebra, these associated primes are finite in number by~\cite[Remark~3.7~(i)]{Lyubeznik:Invent}.

If $\frakp$ is generated by a prime integer $p$, the exactness of
\[
\CD
H^{k-1}_\fraka(R/pR)@>d>>H^k_\fraka(R)@>p>>H^k_\fraka(R)
\endCD
\]
shows that an associated prime of $H^k_\fraka(R)$ that contains $p$ is an associated prime of
\[
\ker(p)\ =\ \image (d)\,.
\]
It thus suffices to show that $\image(d)$ has finitely many associated primes as an $R$-module, or, equivalently, as an $R/pR$-module.

Recall that~\eqref{eqn:image:d} is a surjection of $\calD(R/pR)$-modules. By \cite[Corollary~5.10]{Lyubeznik:Crelle}, the module $H^{k-1}_\fraka(R/pR)$ has finite length as a~$\calD(R/pR)$-module, and hence so does $\image(d)$. The associated primes of $\image(d)$ are among the minimal primes of its simple $\calD(R/pR)$-module subquotients; it thus suffices to show that each simple $\calD(R/pR)$-module has a unique associated prime. Indeed, let $M$ be a simple $\calD(R/pR)$-module, and $\frakp$ a maximal element of $\Ass_{R/pR}M$. Then $H^0_\frakp(M)$ is a $\calD(R/pR)$-submodule of $M$, and hence it must equal $M$. But $\frakp$ is maximal in $\Ass_{R/pR}M$, so it is the unique associated prime of $M$.
\end{proof}

We conclude the section with two examples:

\begin{example}
Given a finite set of prime integers $S$, there exists a polynomial ring $R$ over~$\ZZ$, a monomial ideal $\fraka$ in $R$, and an integer $k$, such that $H^k_\fraka(R)$ has $p$-torsion if and only if~$p\in S$; see~\cite[Example~5.11]{SW:Crelle}.
\end{example}

\begin{example}
Let $E$ be an elliptic curve in $\PP^2_\QQ$. Consider the Segre embedding of~$E\times\PP^1_\QQ$ in $\PP_\QQ^5$, and let $\fraka$ be a lift of the defining ideal to $R=\ZZ[x_0,\dots,x_5]$, i.e.,
\[
\Proj\big(R/\fraka\otimes_\ZZ\QQ\big)\ =\ E\times\PP^1_\QQ\,.
\]
By~\cite[page~75]{HS} or~\cite[page~219]{Lyubeznik:Compositio}, the module $H^4_\fraka(R/pR)$ is zero for infinitely many prime integers $p$ (corresponding to $E\mod p$ being supersingular) and nonzero for infinitely many $p$ (corresponding to $E\mod p$ being ordinary); see also~\cite[Corollary~2.2]{SW:Contemp}. Thus,
\[
\CD
H^4_I(R)@>p>>H^4_I(R)
\endCD
\]
is surjective for infinitely many primes $p$, and also not surjective for infinitely many~$p$. Theorem~\ref{theorem:main} implies that the map is injective for all but finitely many primes $p$.
\end{example}

\section{Smooth algebras over a Dedekind domain}
\label{section:dedekind}

We indicate how Theorem~\ref{theorem:main} extends to algebras that are smooth over the ring of integers of a number field; first, the local version:

\begin{theorem}
\label{theorem:dvr}
Let $(V,uV)$ be a discrete valuation ring of mixed characteristic. Let $R$ be a $V$-algebra that is either smooth over $V$, or a formal power series ring over $V$.

Let $\fraka$ be an ideal of $R$ generated by elements $\bsf$.
\begin{enumerate}[\quad\rm(1)]
\item If $u$ is a nonzerodivisor on $H^k(\bsf;R)$, then it is a nonzerodivisor on $H^k_\fraka(R)$.

\item The $R$-module $H_\fraka^k(R)$ has finitely many associated prime ideals.
\end{enumerate}
\end{theorem}

\begin{proof}
We first reduce to the case where $V$ has a perfect residue field: There exists a discrete valuation ring $(V',uV')$ such that $V'/uV'$ is a perfect field, and $V\to V'$ is faithfully flat, see, for example, \cite[Chapter~IX,~Appendice~2]{Bourbaki}. Take $R'$ to be either $R\otimes_V V'$ or a formal power series ring over $V'$, in the respective cases; note that if $R$ is smooth over $V$, then $R'$ is smooth over $V'$. In either case, $R'$ is faithfully flat over $R$, and it suffices to prove the assertions of the theorem for the ring $R'$.

We may thus assume that $V/uV$ is a perfect field; it follows that $R/uR$ is an $F$-finite regular ring. As before, the exact sequence
\[
\CD
0@>>>R@>u>>R@>>>R/uR@>>>0
\endCD
\]
induces the commutative diagram with exact rows:
\[
\CD
@>>>H^{k-1}(\bsf;R)@>\pi>>H^{k-1}(\bsf;R/uR)@>>>H^k(\bsf;R)@>u>>H^k(\bsf;R)@>>>\\
@. @V\alpha' VV @VV\alpha V @VVV @VVV \\
@>>>H^{k-1}_\fraka(R)@>\phi>>H^{k-1}_\fraka(R/uR)@>d>>H^k_\fraka(R)@>u>>H^k_\fraka(R)@>>>
\endCD
\]
The bottom row is a complex of $\calD(R,V)$-modules; specifically, $\image(\phi)$ and $\image(d)$ are $\calD(R,V)$-modules. Since they are annihilated by $u$, they are also modules over the ring~$\calD(R,V)/u\calD(R,V)$. If $R$ is smooth over $V$, then Lemma~\ref{lemma:dmodp} gives
\[
\calD(R,V)/u\calD(R,V)\ =\ \calD(R/uR,\,V/uV)\,;
\]
the same holds when $R$ is a ring of formal power series over $V$ by~\eqref{equation:dmodp}. Moreover, since~$V/uV$ is a perfect field, one has
\[
\calD(R/uR\,,\,V/uV)\ =\ \calD(R/uR)\,.
\]
The remainder of the proof now proceeds analogous to that of Theorem~\ref{theorem:integer}.\footnote{In the formal power series case one cannot use \cite[Theorem~2.4]{Lyubeznik:Invent} for a proof of the finiteness of the prime ideals not containing $u$ because the ring is not finitely generated over a field. A proof of this remains the same as in \cite[pp.~5880 (from line -5)--5882]{Lyubeznik:Comm}; but our proof in the formal power series case of the finiteness of the primes containing $u$ is much simpler than in \cite{Lyubeznik:Comm}.}
\end{proof}

As a consequence, we recover the following result of Lyubeznik, \cite[Theorem~1]{Lyubeznik:Comm}:

\begin{corollary}
Let $R$ be an unramified regular local ring of mixed characteristic, or, more generally, assume that the completion of $R$ is a formal power series ring over a discrete valuation ring of mixed characteristic.

Then each local cohomology module $H^k_\fraka(R)$ has finitely many associated prime ideals.
\end{corollary}

\begin{proof}
One reduces to the case where $R$ is a formal power series ring over a discrete valuation ring of mixed characteristic; the result then follows from Theorem~\ref{theorem:dvr}.
\end{proof}

\begin{theorem}
\label{theorem:dedekind}
Let $A$ be the ring of integers of a number field, or, more generally, a Dedekind domain such that for each height one prime ideal $\frakp$ of $A$, the local ring $A_\frakp$ has mixed characteristic. Let $R$ be a smooth $A$-algebra. Then each local cohomology module~$H^k_\fraka(R)$ has finitely many associated prime ideals.
\end{theorem}

\begin{proof}
Fix a generating set $\bsf$ for $\fraka$. The $R$-module $H^k(\bsf;R)$ has finitely many associated prime ideals; let $\frakp_1,\dots,\frakp_m$ be the contractions of these to the ring $A$. Let~$\frakp$ be a height one prime of $A$ that differs from the $\frakp_i$. We claim that $\frakp$ is not an associated prime of $H^k_\fraka(R)$, viewed as an $A$-module.

Indeed, if it is, then $\frakp A_\frakp$ is an associated prime of $H^k_\fraka(R_\frakp)$ as an $A_\frakp$-module; but then, by Theorem~\ref{theorem:dvr}\!~(1), $\frakp A_\frakp$ is an associated prime of $H^k(\bsf;R_\frakp)$ as an $A_\frakp$-module, implying that~$\frakp$ is an associated prime of~$H^k(\bsf;R)$ as an $A$-module, which is false. This proves the~claim.

Hence $H^k_\fraka(R)$ has finitely many associated primes as an $A$-module. By Theorem~\ref{theorem:dvr}\!~(2), there are finitely many elements of $\Ass_R H^k_\fraka(R)$ lying over each element of $\Ass_A H^k_\fraka(R)$.
\end{proof}



\begin{thebibliography}{HKM}
\bibitem[ABL]{ABL}
J.~\`Alvarez Montaner, M.~Blickle, and G.~Lyubeznik, \emph{Generators of $D$-modules in characteristic $p>0$}, Math. Res. Lett.~\textbf{12} (2005), 459--473.

\bibitem[Bo]{Bourbaki}
N.~Bourbaki, \emph{\'El\'ements de math\'ematique, Alg\`ebre commutative, Chapitres 8 et 9}, Springer, Berlin, 2006.

\bibitem[Gr]{EGA4}
A.~Grothendieck, \emph{\'El\'ements de g\'eom\'etrie alg\'ebrique IV, \'Etude locale des sch\'emas et des morphismes de sch\'emas IV}, Inst. Hautes \'Etudes Sci. Publ. Math.~\textbf{32} (1967), 5--361.

\bibitem[HaS]{HS}
R.~Hartshorne and R.~Speiser, \emph{Local cohomological dimension in characteristic $p$}, Ann. of Math.~(2) \textbf{105} (1977), 45--79.

\bibitem[Hu]{Huneke:Sundance}
C.~Huneke, \emph{Problems on local cohomology}, in: Free resolutions in commutative algebra and algebraic geometry (Sundance, Utah, 1990), 93--108, Res. Notes Math.~\textbf{2}, Jones and Bartlett, Boston, MA, 1992.

\bibitem[HuS]{HS:TAMS}
C.~Huneke and R.~Sharp, \emph{Bass numbers of local cohomology modules}, Trans. Amer. Math. Soc.~\textbf{339} (1993), 765--779.

\bibitem[Ka]{Katzman}
M.~Katzman, \emph{An example of an infinite set of associated primes of a local cohomology module}, J. Algebra~\textbf{252} (2002), 161--166.

\bibitem[Ly1]{Lyubeznik:Invent}
G.~Lyubeznik, \emph{Finiteness properties of local cohomology modules (an application of $D$-modules to commutative algebra)}, Invent. Math.~\textbf{113} (1993), 41--55.

\bibitem[Ly2]{Lyubeznik:Crelle}
G.~Lyubeznik, \emph{$F$-modules: applications to local cohomology and $D$-modules in characteristic $p>0$}, J. Reine Angew. Math.~\textbf{491} (1997), 65--130.

\bibitem[Ly3]{Lyubeznik:Comm}
G.~Lyubeznik, \emph{Finiteness properties of local cohomology modules for regular local rings of mixed characteristic: the unramified case}, Special issue in honor of Robin Hartshorne, Comm. Alg.~\textbf{28} (2000), 5867--5882.

\bibitem[Ly4]{Lyubeznik:Compositio}
G.~Lyubeznik, \emph{On the vanishing of local cohomology in characteristic $p>0$}, Compos. Math.~\textbf{142} (2006), 207--221.

\bibitem[Nu]{Nunez}
L.~N\'u\~nez-Betancourt, \emph{On certain rings of differentiable type and finiteness properties of local cohomology}, J. Algebra \textbf{379} (2013), 1--10.

\bibitem[Si]{Singh:MRL}
A.~K.~Singh, \emph{$p$-torsion elements in local cohomology modules}, Math. Res. Lett.~\textbf{7} (2000), 165--176.

\bibitem[SS]{SS:IMRN}
A.~K.~Singh and I.~Swanson, \emph{Associated primes of local cohomology modules and of Frobenius powers}, Int. Math. Res. Not.~\textbf{33} (2004), 1703--1733.

\bibitem[SW1]{SW:Contemp}
A.~K.~Singh and U.~Walther, \emph{On the arithmetic rank of certain Segre products}, Contemp. Math.~\textbf{390} (2005) 147--155.

\bibitem[SW2]{SW:Crelle}
A.~K.~Singh and U.~Walther, \emph{Bockstein homomorphisms in local cohomology}, J. Reine Angew. Math.~\textbf{655} (2011), 147--164.

\end{thebibliography}
\end{document}